\documentclass[11pt]{amsart}%
\usepackage{amsfonts}
\usepackage{amsmath}
\usepackage{amssymb}
\usepackage{graphicx}%
\newtheorem{theorem}{Theorem}
\theoremstyle{plain}

\newtheorem{lemma}{Lemma}

\newtheorem{proposition}{Proposition}
\newtheorem{remark}{Remark}

\numberwithin{equation}{section}
\begin{document}
\title[Existence of nodal solutions...]{\textbf{Nodal solutions for fourth order elliptic equations with critical
exponent on compact manifolds }}
\author{Mohamed Bekiri}
\address{Mohamed Bekiri Faculty Of Sciences, Mathematics Dept; University Abou-Bakr
Belka\^{\i}d; Tlemcen. Algeria}
\email{bekiri03@yahoo.fr}
\author{Mohammed Benalili }
\address{Mohammed Benalili Faculty Of Sciences, Mathematics Dept; University Abou-Bakr
Belka\^{\i}d; Tlemcen. Algeria}
\email{m\_benalili@mail.univ-tlemcen.dz}
\thanks{}
\date{October 22, 2016}
\dedicatory{ }
\begin{abstract}
Using a variational method, we prove the existence of nodal solutions to
prescribed scalar $Q$- curvature type equations on compact Riemannian
manifolds with boundary. These equations are fourth-order elliptic equations
with critical Sobolev growth .

\end{abstract}
\maketitle

\section{Introduction}

The Paneitz operator was discovered by Paneitz on $4$-dimension manifolds (
see \cite{23}) and extended by Branson to higher dimensions $\left(
n\geq5\right)  $ ( see \cite{5}) is given by
\begin{equation}
P_{g}^{n}\left(  u\right)  =\Delta_{g}^{2}u-\operatorname{div}_{g}\left(
\frac{\left(  n-2\right)  ^{2}+4}{2\left(  n-1\right)  \left(  n-2\right)
}R_{g}g-\frac{4}{n-2}Ric_{g}\right)  ^{\#}du+\frac{n-4}{2}Q_{g}^{n}u
\label{eqn1}%
\end{equation}
where $\Delta_{g}u=-\operatorname{div}_{g}\left(  \nabla u\right)  $ is the
Laplace-Beltrami operator, $R_{g}$ and $Ric_{g}$ denote the scalar curvature
and Ricci curvature respectively, the symbol $\#$ stands for the musical
isomorphism and $Q_{g}^{n}$ is the $Q$-curvature which is expressed by%

\[
Q_{g}^{n}=\frac{1}{2\left(  n-1\right)  }\Delta_{g}R+\frac{n^{3}%
-4n^{2}+16n-16}{8\left(  n-1\right)  ^{2}\left(  n-2\right)  ^{2}}R^{2}%
-\frac{2}{\left(  n-2\right)  ^{2}}\left\vert Ric_{g}\right\vert ^{2}.
\]
Geometrically the $Q$-curvature can be interpreted as the analogue of the
scalar curvature for the conformal Laplacian.

Let $\left(  M,g\right)  $ be a smooth Riemannian compact manifold with
boundary and of dimension $\left(  n\geq5\right)  $. We let $A$ be a smooth
symmetric $\left(  2,0\right)  $-tensor on $M$ and $a\in C^{\infty}\left(
M\right)  .$ The Paneitz-Branson type operator with general coefficients is an
operator of the form%
\begin{equation}
P_{g}u=\Delta_{g}^{2}u-\operatorname{div}_{g}\left(  A\left(  \nabla u\right)
^{\#}\right)  +au. \label{eqn2}%
\end{equation}
If the coefficients are constant it reads as follows
\begin{equation}
P_{g}^{n}u=\Delta_{g}^{2}u+\alpha\Delta_{g}u+\beta u\text{.} \label{eqn3}%
\end{equation}
In particular on Einstein manifolds, it is reduced to%
\[
P_{g}^{n}\left(  u\right)  =\Delta_{g}^{2}u+\alpha_{n}\Delta_{g}u+\beta_{n}u
\]
where
\[
\alpha_{n}=\frac{n^{2}-2n-4}{2n\left(  n-1\right)  }R_{g}\text{ and }\beta
_{n}=\frac{\left(  n-4\right)  \left(  n^{2}-4\right)  }{16\left(  n-1\right)
^{2}}R_{g}^{2}\text{.}%
\]
A nice property of the Paneitz-Branson operator is its conformal covariance
i.e. if $\tilde{g}=\varphi^{\frac{4}{n-4}}g$ is a conformal metric to $g$
where $\varphi$ stands for a smooth positive function, then for all $u\in
C^{\infty}\left(  M\right)  $ we have%
\[
P_{g}^{n}\left(  u\varphi\right)  =\varphi^{\frac{n+4}{n-4}}P_{\tilde{g}}%
^{n}\left(  u\right)  \text{.}%
\]
In particular, by setting $u=1$, we obtain%
\[
P_{g}^{n}\varphi=\frac{n-4}{2}Q_{\tilde{g}}^{n}\varphi^{\frac{n+4}{n-4}%
}\text{.}%
\]

In this work seek the existence of a real number $\lambda$ and a nodal
solution $u$ of the following Dirichlet problem%
\begin{equation}
\left\{
\begin{array}
[c]{ll}%
P_{g}u=\lambda f\left\vert u\right\vert ^{2^{\sharp}-2}u & \text{in }M\\
u=\phi_{1}\text{, }\partial_{\nu}u=\phi_{2} & \text{on }\partial M
\end{array}
\right.  \label{eqn4}%
\end{equation}
where $P_{g}$ is the Paneitz-Branson type operator defined by $\left(
\text{\ref{eqn2}}\right)  $, $\phi_{1},\phi_{2}\in C^{\infty}\left(  \partial
M\right)  $ are boundary data and $2^{\sharp}=\frac{2n}{n-4}$ is the Sobolev
critical exponent. When $\phi_{1}$\ changes sign, $u$ is called a nodal
solution of the equation $\left(  \text{\ref{eqn4}}\right)  $.

Many authors have paid attention to the study of nodal solutions to the
second-order Dirichlet problem and several works were carried out; we quote
some of them :

In 1990, F.V. Atkinson, H. Brezis and L.A. Peletier \cite{1}, studied the
existence of nodal solutions of the following eigenvalue problem
\[
\left\{
\begin{array}
[c]{ll}%
-\Delta u=\lambda u+\left\vert u\right\vert ^{p}u & \text{in }B\\
u\not \equiv 0 & \text{in }B\\
u=0 & \text{on }\partial B
\end{array}
\right.
\]
where $B$ is the unit ball of $%
\mathbb{R}
^{n}$ $\left(  n\geq3\right)  $, $p=\frac{n+2}{n-2},$ and $\lambda$ is a
positive real number.

In 1994, E. Hebey and M. Vaugon \cite{18} obtained existence and multiplicity
results of nodal solutions to the problem%
\[
\left\{
\begin{array}
[c]{ll}%
-\Delta_{g}u+au=f\left\vert u\right\vert ^{\frac{4}{n-2}}u & \text{in }%
\Omega\\
u\not \equiv 0 & \text{in }\Omega\\
u=0 & \text{on }\partial\Omega
\end{array}
\right.
\]
where $\Omega$ is a smooth bounded domain of $R^{n}$ $\left(  n\geq3\right)
$, $a$, $f\in C^{\infty}\left(  \bar{\Omega}\right)  $ and $g$ is a Riemannian
metric defined in a neighborhood of $\bar{\Omega}.$ Moreover, they considered
this problem when the zero data on the boundary is replaced by a non-zero one.

In \cite{19} D. Holcman obtained by variational method nodal solutions to the
following problem%
\[
\left\{
\begin{array}
[c]{ll}%
\Delta_{g}u+au=\lambda f\left\vert u\right\vert ^{\frac{4}{n-2}}u & \text{in
}M\\
u=\phi & \text{on }\partial M
\end{array}
\right.
\]
where $\left(  M,g\right)  $ is a smooth compact Riemannian manifold with
boundary and of dimension $n\geq3$, $a$, $f\in C^{\infty}\left(  M\right)  $,
$\phi\in C^{\infty}\left(  \partial M\right)  $ is a changing sign function
and where $\lambda$ stands for a real positive number.

In 2002, Z. Djadli and A. Jourdain \cite{11} studied nodal solutions for
scalar curvature type equations with perturbations.

With the discovery of the Paneitz invariant and consequently the introduction
of the Q-curvature studies similar to those conducted for the scalar curvature
were undertaken by several people. The latter consists in the search for
solutions to fourth-order elliptic equations with critical exponents in the
sense of the Sobolev inclusions. We refer the reader to the works by Benalili
\cite{4}, Chang \cite{8}, Chang-Yang \cite{9}, Djadli-Hebey-Ledoux \cite{10},
Djadli-Malchiodi-Ould Ahmedou \cite{12}, \cite{13}, Esposito-Robert \cite{15},
Felli \cite{16}, Hebey-Robert \cite{17}, Lin \cite{21}, Robert \cite{24},
\cite{25}, to quote only few of them.

We denote by $H_{2,0}^{2}\left(  M\right)  $ the standard Sobolev space which
is the completion of the space $C_{c}^{\infty}\left(  M\right)  $ with respect
to the norm
\[
\left\Vert u\right\Vert _{H_{2}^{2}\left(  M\right)  }^{2}=\sum\nolimits_{i=0}%
^{i=2}\left\Vert \nabla^{i}u\right\Vert _{2}^{2}%
\]
where $C_{c}^{\infty}\left(  M\right)  $ denotes the space of smooth functions
with compact supports in $M$.

This standard is equivalent to the following Hilbert norm%
\[
\left\Vert u\right\Vert ^{2}=\left\Vert \Delta u\right\Vert _{2}%
^{2}+\left\Vert \nabla u\right\Vert _{2}^{2}+\left\Vert u\right\Vert _{2}%
^{2}.
\]
We know from the Sobolev theorem that the inclusion $H_{2,0}^{2}\left(
M\right)  $ is continuously embedded in $L^{2^{\sharp}}\left(  M\right)  $,
but it is not compactly embedded.\newline Let $K_{0}$ be the best constant in
the Euclidean Sobolev inequality:%
\[
\left\Vert u\right\Vert _{2^{\sharp}}^{2}\leq K_{0}\left\Vert \Delta
u\right\Vert _{2}^{2}%
\]
for all $u\in C_{c}^{\infty}\left(  R\right)  $. The value of $K_{0}$ was
computed by Edmunds-Fortunato-Janelli \cite{14}, Lieb \cite{20} and Lions
\cite{22}. They obtained that%
\[
\frac{1}{K_{0}}=\frac{n\left(  n^{2}-4\right)  \left(  n-4\right)  \omega
_{n}^{\frac{4}{n}}}{16}%
\]
where $\omega_{n}$ denotes the volume of the Euclidean unit sphere $\left(
\mathbb{S}^{n},h\right)  $, endowed with its standard metric.

Following standard terminology, we say that the Paneitz-Branson type operator
$P_{g}=\Delta_{g}^{2}-\operatorname{div}_{g}\left(  A\left(  \nabla.\right)
^{\#}\right)  +a$ is coercive if there exists $\Lambda>0,$ such that for any
$u\in H_{2,0}^{2}\left(  M\right)  $%
\[
\int\nolimits_{M}uP_{g}udv_{g}\geq\Lambda\left\Vert u\right\Vert _{H_{2,0}%
^{2}\left(  M\right)  }^{2}\text{.}%
\]
In this paper we extend some results obtained by D. Holcman \cite{19} for
equations of the Yamabe type to the equations containing the operator of the
Paneitz-Branson type. The main results of our study are summarized in the
following theorem which is the version with boundary of Theorem 3 of
Episoto-Robert \cite{15}.

\begin{theorem}
\label{th1} Let $\left(  M,g\right)  $ be a compact Riemannian manifold of
dimension $n\geq6$ with boundary $\partial M{\neq\emptyset}$. Let $A$ be a
smooth symmetric $\left(  2,0\right)  $-tensor and $a,$ $f\in C^{\infty
}\left(  M\right)  $, $f>0$ and $x_{0}\in M$ such that $f\left(  x_{0}\right)
=\max_{M}f.$ We assume that the operator $P_{g}=\Delta_{g}^{2}%
-\operatorname{div}_{g}\left(  A\left(  \nabla.\right)  ^{\#}\right)  +a$ is
coercive and%
\[
8\left(  n-1\right)  Tr_{g}A\left(  x_{0}\right)  +\left(  n-6\right)  \left(
n-4\right)  \left(  n+2\right)  \frac{\Delta f\left(  x_{0}\right)  }{f\left(
x_{0}\right)  }-4\left(  n^{2}-2n-4\right)  R_{g}\left(  x_{0}\right)
<0\text{.}%
\]
Then there exists a positive real number $\lambda$ and a nontrivial solution
$w=u-h\in H_{2,0}^{2}\left(  M\right)  \cap C^{4}\left(  M\right)  $ of the
equation $\left(  \text{\ref{eqn4}}\right)  $ which is a minimizer of the
functional $I$ defined on $H_{2,0}^{2}\left(  M\right)  $ by $I\left(
u\right)  =\int\nolimits_{M}\left(  \left(  \Delta_{g}u\right)  ^{2}+A\left(
\left(  \nabla u\right)  ^{\#},\left(  \nabla u\right)  ^{\#}\right)
+au^{2}\right)  dv_{g}$ under the constraint $\int_{M}f\left\vert
u+h\right\vert ^{2^{\sharp}}dv_{g}=\gamma$, where $h$ denotes the unique
solution of the problem%
\[
\left\{
\begin{array}
[c]{ll}%
\Delta_{g}^{2}h-\operatorname{div}_{g}\left(  A\left(  \nabla h\right)
^{\#}\right)  +ah=0 & \text{in }M\\
h=\phi_{1}\text{, }\partial_{\nu}h=\phi_{2} & \text{on }\partial M
\end{array}
\right.
\]
and where $\phi_{1}$, $\phi_{2}$ are smooth functions on the boundary
$\partial M$ with $\phi_{1}$ is of changing sign.
\end{theorem}

\begin{remark}
Since the function $\phi_{1}$ is of changing sign, the solutions obtained in
Theorem \ref{th1} are nodal.
\end{remark}

The proof of theorem \ref{th1} proceeds in several steps: in a first section
we use the approach developed by H. Yamabe in \cite{27}: we construct a
minimizing sequence of solutions to the subcritical equations. In the second
one, we show that this sequence converges weakly to a solution of the critical
equation when the subcritical exponent tends to the critical Sobolev exponent.
In the third section we show that under some conditions, we obtain a non
trivial solution of the critical equation. The last section is devoted to test
functions which verify the conditions assumed in the generic Proposition
\ref{prop1}.

\section{Subcritical solutions}

We shall make use of the following Sobolev inequality on compact manifolds
with boundary obtained in a more general context (see \cite{6})

\begin{lemma}
\label{lem1} (\cite{6}) Let ($M,g$) be a smooth compact Riemannian manifold
with boundary of dimension $n\geq5.$ Then given $\varepsilon>0$, then there
exists a positive constant $B$ which depends on $M$, $g$, $\varepsilon$ such
that%
\[
\left\Vert u\right\Vert _{2^{\sharp}}^{2}\leq\left(  K_{0}^{2}+\varepsilon
\right)  \left\Vert \Delta u\right\Vert _{2}^{2}+B\left\Vert u\right\Vert
_{2}^{2}%
\]
for all $u\in H_{2,0}^{2}\left(  M\right)  $.
\end{lemma}

First, we state the following useful proposition.

\begin{proposition}
Let $\left(  M,g\right)  $ be a smooth Riemannian compact manifold with smooth
boundary and of dimension $n\geq6.$ We assume that the operator $P_{g}$
defined in $\left(  \text{\ref{eqn2}}\right)  $ is coercive. Then there exists
a unique $h\in C^{4}\left(  M\right)  $ solution of the following problem
\begin{equation}
\left\{
\begin{array}
[c]{ll}%
P_{g}h=0 & \text{in }M\\
h=\phi_{1}\text{, }\partial_{\nu}h=\phi_{2} & \text{on }\partial M
\end{array}
\right.  . \label{eqn5}%
\end{equation}

\end{proposition}

\begin{proof}
Since the operator $P_{g}$ is coercive, the existence and the uniqueness of
the solution are guaranteed by the Lax-Milgram's theorem. The regularity
follows from general regularity theory.
\end{proof}

Let $w=u-h$, we first notice that $u\in C^{4}\left(  M\right)  $ is solution
of the equation $\left(  \text{\ref{eqn4}}\right)  $ if and only if $w\in
C^{4}\left(  M\right)  $ is solution of
\begin{equation}
\left\{
\begin{array}
[c]{ll}%
P_{g}w=\lambda f\left\vert w+h\right\vert ^{2^{\sharp}-2}\left(  w+h\right)  &
\text{in }M\\
w=\partial_{\nu}w=0 & \text{on }\partial M
\end{array}
\right.  . \label{eqn6}%
\end{equation}
The associated subcritical problem to \ref{eqn6} is then
\begin{equation}
\left\{
\begin{array}
[c]{ll}%
P_{g}w=\lambda f\left\vert w+h\right\vert ^{q-2}\left(  w+h\right)  & \text{in
}M\\
w=\partial_{\nu}w=0 & \text{on }\partial M
\end{array}
\right.  \label{eqn7}%
\end{equation}
where $2<q<2^{\sharp}$.\newline The functional associated to equation $\left(
\text{\ref{eqn7}}\right)  $ is defined on $H_{2,0}^{2}\left(  M\right)  $ by
\[
I\left(  w\right)  =\int\nolimits_{M}wP_{g}wdv_{g}=\int\nolimits_{M}\left(
\left(  \Delta w\right)  ^{2}+A\left(  \left(  \nabla w\right)  ^{\#},\left(
\nabla w\right)  ^{\#}\right)  +aw^{2}\right)  dv_{g}\text{.}%
\]
Denote by%
\[
\mu_{\gamma,q}:=\inf_{w\in\mathcal{H}_{q}}I\left(  w\right)  \text{ and
\ }\mathcal{H}_{q}=\left\{
\begin{array}
[c]{cc}%
w\in H_{2,0}^{2}\left(  M\right)  \text{ such that} & \int\nolimits_{M}%
f\left\vert w+h\right\vert ^{q}dv_{g}=\gamma
\end{array}
\right\}
\]
where $\gamma$ is a constant such that
\begin{equation}
\int\nolimits_{M}f\left\vert h\right\vert ^{2^{\sharp}}dv_{g}<\gamma\text{.}
\label{eqn8}%
\end{equation}
We state that

\begin{lemma}
For all $\gamma>\int\nolimits_{M}f\left\vert h\right\vert ^{2^{\sharp}}dv_{g}$
and for all $\ 0\leq q<2^{\sharp}$, there exists a real number $\lambda
_{\gamma,q}$ and a function $w_{\gamma,q}\in\mathcal{H}_{q}$ solution of the
problem $\left(  \text{\ref{eqn7}}\right)  $.
\end{lemma}

\begin{proof}
First we show that $\mathcal{H}_{q}$ is not empty under the condition
$\int\nolimits_{M}f\left\vert h\right\vert ^{2^{\sharp}}dv_{g}<\gamma$. To do
this, we set
\[
F\left(  t\right)  =\int\nolimits_{M}f\left\vert t\psi_{1}+h\right\vert
^{q}dv_{g}%
\]
where $\psi_{1}$ is the eigenfunction corresponding to the first eigenualue
$\lambda_{1}$ of $\Delta_{g}^{2}$ i.e.%
\[
\left\{
\begin{array}
[c]{ll}%
\Delta_{g}^{2}\psi_{1}=\lambda_{1}\psi_{1} & \text{in }M\\
\psi_{1}=\partial_{\nu}\psi_{1}=0 & \text{on }\partial M
\end{array}
\right.  \text{.}%
\]
It is clear that
\[
F\left(  0\right)  =\int\nolimits_{M}f\left\vert h\right\vert ^{q}%
dv_{g}<\gamma\text{ }%
\]
and%
\[
\text{ }\lim\limits_{t\rightarrow+\infty}F\left(  t\right)  =+\infty\text{.}%
\]
Prom the continuity of $F$, there exists $t_{q}>0$ such that
\[
F\left(  t_{q}\right)  =\int\nolimits_{M}f\left\vert t_{q}\psi_{1}%
+h\right\vert ^{q}dv_{g}=\gamma\text{.}%
\]
So $t_{q}\psi_{1}\in\mathcal{H}_{q}$.

Secondly we check that $\mu_{\gamma,q}$ is finite. Let $u\in H_{2,0}%
^{2}\left(  M\right)  $. Since the tensor $A$ is smooth, there exists a
constant $C>0$ such that%
\begin{equation}
\left\vert \int\nolimits_{M}A\left(  \left(  \nabla u\right)  ^{\#},\left(
\nabla u\right)  ^{\#}\right)  dv_{g}\right\vert \leq C\int\nolimits_{M}%
\left\vert \nabla u\right\vert ^{2}dv_{g}\text{.} \label{eqn9}%
\end{equation}

By interpolation $\left(  \text{see Aubin \cite{3} page 93}\right)  ,$ for any
$\eta>0$ there exists \newline$C\left(  \eta\right)  >0$ such that for any
$u\in H_{2,0}^{2}\left(  M\right)  $%
\begin{equation}
\left\Vert \nabla u\right\Vert _{2}^{2}\leq\eta\left\Vert \Delta u\right\Vert
_{2}^{2}+C\left(  \eta\right)  \left\Vert u\right\Vert _{2}^{2}\text{.}
\label{eqn10}%
\end{equation}

Plugging $\left(  \text{\ref{eqn10}}\right)  $ into $\left(  \text{\ref{eqn9}%
}\right)  $ we get
\begin{equation}
\left\vert \int\nolimits_{M}A\left(  \left(  \nabla u\right)  ^{\#},\left(
\nabla u\right)  ^{\#}\right)  dv_{g}\right\vert \leq C\left(  \eta\right)
\left\Vert \Delta u\right\Vert _{2}^{2}+C^{\prime}\left(  \eta\right)
\left\Vert u\right\Vert _{2}^{2}\text{.} \label{eqn 10'}%
\end{equation}

Considering inequality (\ref{eqn 10'}) and the expression of $I$, we deduce
that
\begin{equation}
I\left(  u\right)  \geq\left(  1-C\left(  \eta\right)  \right)  \left\Vert
\Delta u\right\Vert _{2}^{2}-\left(  C^{\prime}\left(  \eta\right)
+\left\Vert a\right\Vert _{\infty}\right)  \left\Vert u\right\Vert _{2}^{2}
\label{eqn11}%
\end{equation}

where $\left\Vert .\right\Vert _{\infty}$ is the supremum norm.

On the other hand, for all $u\in\mathcal{H}_{q}$, by H\"{o}lder's inequality
we have%
\begin{equation}
\left\Vert u\right\Vert _{2}^{2}\leq Vol_{g}\left(  M\right)  ^{1-\frac{2}{q}%
}\left(  \left(  \min_{M}f\right)  ^{-\frac{1}{q}}\gamma^{\frac{1}{q}%
}+\left\Vert h\right\Vert _{q}\right)  ^{2}\text{.} \label{eqn12}%
\end{equation}

Plugging $\left(  \ref{eqn12}\right)  $ in $\left(  \ref{eqn11}\right)  $ and
taking $\eta$ small enough, we get%
\[
1-C\left(  \eta\right)  =C_{1}>0\text{.}%
\]
So%
\begin{equation}
I\left(  u\right)  \geq C_{1}\left\Vert \Delta u\right\Vert _{2}^{2}+C_{2}
\label{eqn13}%
\end{equation}
where $C_{1}$ some positive constant and $C_{2}$ is a constant independent of
$u$.

Let $\left(  w_{i}\right)  $ be a minimizing sequence of the functional $I$ on
$\mathcal{H}_{q}$. Hence for $i$ sufficiently large $I\left(  w_{i}\right)
\leq\mu_{\gamma,q}+1$ and by $\left(  \ref{eqn13}\right)  $ we obtain%
\begin{equation}
\left\Vert \Delta w_{i}\right\Vert _{2}^{2}\leq\frac{1}{C_{1}}\left(
\mu_{\gamma,q}+1-C_{2}\right)  \text{.} \label{13'}%
\end{equation}

From $\left(  \ref{eqn10}\right)  $ and $\left(  \ref{eqn12}\right)  $ we
deduce that $\left\Vert \nabla w_{i}\right\Vert _{2}^{2}$ and $\left\Vert
w_{i}\right\Vert _{2}^{2}$ are bounded, and by (\ref{13'}) the sequence
$\left(  w_{i}\right)  $ is bounded in $H_{2,0}^{2}\left(  M\right)  $. By the
reflexivity of the latter space, there exists subsequence of $\left(
w_{i}\right)  $ still denoted $\left(  w_{i}\right)  $ such that \newline%
\[%
\begin{array}
[c]{ll}%
\left(  a\right)  & w_{i}\rightharpoonup w_{\gamma,q}\text{ weakly in }%
H_{2,0}^{2}\left(  M\right)  \medskip\\
\left(  b\right)  & w_{i}\rightarrow w_{\gamma,q}\text{ strongly in }H_{1}%
^{2}\left(  M\right)  \text{ and }L^{s}\left(  M\right)  \text{ for all
}s<2^{\sharp}\medskip\\
\left(  c\right)  & \left\Vert w_{\gamma,q}\right\Vert _{H_{2,0}^{2}}\leq
\lim\limits_{i}\inf\left\Vert w_{i}\right\Vert _{H_{2,0}^{2}}%
\end{array}
\]
\newline Consequently%
\[%
\begin{array}
[c]{ll}%
I\left(  w_{\gamma,q}\right)  & =\int\nolimits_{M}\left(  \left(  \Delta
w_{\gamma,q}\right)  ^{2}+A\left(  \left(  \nabla w_{\gamma,q}\right)
^{\#},\left(  \nabla w_{\gamma,q}\right)  ^{\#}\right)  +aw_{\gamma,q}%
^{2}\right)  dv_{g}\medskip\\
& \leq\lim_{i}\inf\left\Vert \Delta w_{i}\right\Vert _{2}^{2}+\lim_{i}%
\int\nolimits_{M}A\left(  \left(  \nabla w_{i}\right)  ^{\#},\left(  \nabla
w_{i}\right)  ^{\#}\right)  dv_{g}+\lim_{i}\int\nolimits_{M}aw_{i}^{2}%
dv_{g}\medskip\\
& =\lim_{i}I\left(  w_{i}\right)  =\mu_{\gamma,q}%
\end{array}
\]
\newline Since%
\[
\int\nolimits_{M}f\left\vert w_{\gamma,q}+h\right\vert ^{q}dv_{g}%
=\lim\limits_{i}\int\nolimits_{M}f\left\vert w_{i}+h\right\vert ^{q}%
dv_{g}=\gamma
\]
we obtain%
\[
I\left(  w_{\gamma,q}\right)  =\mu_{\gamma,q}%
\]
so $w_{\gamma,q}$ satisfies%
\[%
\begin{array}
[c]{l}%
\int\nolimits_{M}\left(  \Delta w_{\gamma,q}\Delta\varphi+A\left(  \left(
\nabla w_{\gamma,q}\right)  ^{\#},\left(  \nabla\varphi\right)  ^{\#}\right)
+aw_{\gamma,q}\varphi\right)  dv_{g}\\
=\lambda_{\gamma,q}\int\nolimits_{M}f\left\vert w_{\gamma,q}+h\right\vert
^{q-2}\left(  w_{\gamma,q}+h\right)  \varphi dv_{g}%
\end{array}
\]
for any $\varphi\in H_{2,0}^{2}\left(  M\right)  ;$ where $\lambda_{\gamma,q}$
is the Lagrange multiplier . Hence $w_{\gamma,q}$ is a weak solution of the
equation%
\begin{equation}
\left\{
\begin{array}
[c]{ll}%
\Delta_{g}^{2}w_{\gamma,q}-\operatorname{div}A\left(  \nabla w_{\gamma
,q}\right)  ^{\#}+aw_{\gamma,q}=\lambda_{\gamma,q}f\left\vert w_{\gamma
,q}+h\right\vert ^{q-2}\left(  w_{\gamma,q}+h\right)  & \text{in }M\\
w_{\gamma,q}=\partial_{\nu}w_{\gamma,q}=0 & \text{on }\partial M
\end{array}
\right.  \label{eqn14}%
\end{equation}

Using the boostrap method, we show that $w_{\gamma,q}$ $\in L^{\infty}\left(
M\right)  ,$ so $P_{g}\left(  w_{\gamma,q}\right)  =\Delta_{g}^{2}w_{\gamma
,q}-\operatorname{div}A\left(  \nabla w_{\gamma,q}\right)  ^{\#}+aw_{\gamma
,q}\in L^{s}\left(  M\right)  $ for any $\left(  s<2^{\#}\right)  .$ Since
$P_{g}$ is a fourth order elliptic operator, it follows by a well known
regularity theorem that $w_{\gamma,q}\in C^{0,\alpha}\left(  M\right)  ,$ for
some $\alpha\in\left(  0,1\right)  $; hence $w_{\gamma,q}\in C^{4,\alpha
}\left(  M\right)  $.
\end{proof}

\section{Critical solutions}

In this section we will show the exitence of a non trivial solution of the
critical equation (\ref{eqn4}).

\begin{proposition}
Under the hypothesis $\int\nolimits_{M}f\left\vert h\right\vert ^{q}%
dv_{g}<\gamma$, the sequence $\left(  w_{\gamma,q}\right)  _{q}$ is bounded in
$H_{2,0}^{2}\left(  M\right)  $. The Lagrange multipliers $\lambda_{\gamma,q}$
are strictly positive and the sequence $(\lambda_{\gamma,q})_{q}$ is bounded
when $q$ tends to $2^{\sharp}$.
\end{proposition}

\begin{proof}
Multiplying $\left(  \ref{eqn14}\right)  $ by $w_{\gamma,q}$ and integrating
yield%
\[%
\begin{array}
[c]{ll}%
0\leq\mu_{\gamma,q} & =\int\nolimits_{M}\left(  P_{g}w_{\gamma,q}\right)
w_{\gamma,q}dv_{g}\medskip\\
& =\lambda_{\gamma,q}\int\nolimits_{M}f\left\vert w_{\gamma,q}+h\right\vert
^{q-2}\left(  w_{\gamma,q}+h\right)  w_{\gamma,q}dv_{g}\medskip\\
& =\lambda_{\gamma,q}\int\nolimits_{M}f\left\vert w_{\gamma,q}+h\right\vert
^{q-2}\left(  w_{\gamma,q}+h\right)  \left(  w_{\gamma,q}+h-h\right)
dv_{g}\medskip\\
& =\lambda_{\gamma,q}\left(  \gamma-\int\nolimits_{M}f\left\vert w_{\gamma
,q}+h\right\vert ^{q-2}\left(  w_{\gamma,q}+h\right)  hdv_{g}\right)
\end{array}
\]
By H\"{o}lder's inequality we get%
\[
\int\nolimits_{M}f\left\vert w_{\gamma,q}+h\right\vert ^{q-2}\left(
w_{\gamma,q}+h\right)  hdv_{g}\leq\gamma^{1-\frac{1}{q}}\left(  \int
\nolimits_{M}f\left\vert h\right\vert ^{q}\right)  ^{\frac{1}{q}}<\gamma.
\]
We deduce that $\lambda_{\gamma,q}\geq0$. Moreover if $\lambda_{\gamma,q}=0$,
then $w_{\gamma,q}$ $=0$. Hence a contradiction with the fact that
$w_{\gamma,q}\in\mathcal{H}_{q}$ and $\int\nolimits_{M}f\left\vert
h\right\vert ^{q}dv_{g}<\gamma$.

Now, we prove that the sequence $\left(  w_{\gamma,q}\right)  _{q}$ is bounded
in $H_{2,0}^{2}\left(  M\right)  .$\newline Let $\psi_{1}$ be an eigenfunction
of $\Delta_{g}^{2}$ corresponding to the eigenvalue $\lambda_{1}$ such that%
\[
\left\{
\begin{array}
[c]{ll}%
\Delta_{g}^{2}\psi_{1}=\lambda_{1}\psi_{1} & \text{in }M\\
\psi_{1}=\partial_{\nu}\psi_{1}=0 & \text{on }\partial M\\
\int\nolimits_{M}\psi_{1}^{2}dv_{g}=1 &
\end{array}
\right.  .
\]
\newline Let%
\[
F\left(  t,q\right)  =\int\nolimits_{M}f\left\vert t\psi_{1}+h\right\vert
^{q}dv_{g}%
\]
In the last section, we obtain $t_{q}\psi_{1}\in\mathcal{H}_{q}.$ Hence%
\[
F\left(  t_{q},q\right)  =\int\nolimits_{M}f\left\vert t_{q}\psi
_{1}+h\right\vert ^{q}dv_{g}=\gamma.
\]
\newline In the following we show that $\frac{\partial F}{\partial t}\left(
t_{q},q\right)  \neq0.$ We proceed by contradiction.

We suppose that $\frac{\partial F}{\partial t}\left(  t_{q},q\right)  =0$ and
obviously, we have%
\[%
\begin{array}
[c]{ll}%
t_{q}\frac{\partial F}{\partial t}\left(  t_{q},q\right)  & =t_{q}%
\int\nolimits_{M}f\left\vert t_{q}\psi_{1}+h\right\vert ^{q-2}\left(
t_{q}\psi_{1}+h\right)  \psi_{1}dv_{g}\medskip\\
& =\int\nolimits_{M}f\left\vert t_{q}\psi_{1}+h\right\vert ^{q-2}\left(
t_{q}\psi_{1}+h\right)  \left(  t_{q}\psi_{1}+h-h\right)  dv_{g}\medskip\\
& =\gamma-\int\nolimits_{M}f\left\vert t_{q}\psi_{1}+h\right\vert
^{q-2}\left(  t_{q}\psi_{1}+h\right)  hdv_{g}=0
\end{array}
\]
Then%
\begin{equation}
\gamma=\int\nolimits_{M}f\left\vert t_{q}\psi_{1}+h\right\vert ^{q-2}\left(
t_{q}\psi_{1}+h\right)  hdv_{g} \label{eqn15}%
\end{equation}
Using the H\"{o}lder's inequality and the hypothesis $\int\nolimits_{M}%
f\left\vert h\right\vert ^{q}dv_{g}<\gamma,$ we deduce that%
\[
\int\nolimits_{M}f\left\vert t_{q}\psi_{1}+h\right\vert ^{q-2}\left(
t_{q}\psi_{1}+h\right)  hdv_{g}\leq\left(  \int\nolimits_{M}f\left\vert
t_{q}\psi_{1}+h\right\vert ^{q}dv_{g}\right)  ^{1-\frac{1}{q}}\left(
\int\nolimits_{M}f\left\vert h\right\vert ^{q}dv_{g}\right)  ^{\frac{1}{q}%
}<\gamma\text{.}%
\]
which contradicts with $\left(  \ref{eqn15}\right)  $.

Since $\frac{\partial F}{\partial t}\left(  t_{q},q\right)  \neq0$ and from
the implicit function theorem, it results that $t_{q}$ is continuous function
of $q$. Hence there exists a constant $C\left(  \gamma\right)  $ independent
of $q$ such that%
\begin{equation}
\int\nolimits_{M}w_{\gamma,q}P_{g}w_{\gamma,q}dv_{g}\leq I\left(  t_{q}%
\psi_{1}\right)  =t_{q}^{2}I\left(  \psi_{1}\right)  \leq C\left(
\gamma\right)  \text{.} \label{eqn16}%
\end{equation}
From the coercivity of $P_{g}$, we find that the sequence $\left(
w_{\gamma,q}\right)  _{q}$ is bounded in $H_{2,0}^{2}\left(  M\right)  $ when
$q$ tends to $2^{\sharp}$.

So, we can extract a subsequence of $\left(  w_{\gamma,q}\right)  $ still
denoted $w_{\gamma,q}$, such that%
\[%
\begin{array}
[c]{ll}%
\left(  a\right)  & w_{\gamma,q}\rightharpoonup w\text{ weakly in }H_{2,0}%
^{2}\left(  M\right)  \text{ as }q\rightarrow2^{\sharp}\medskip\\
\left(  b\right)  & w_{\gamma,q}\rightarrow w\text{ strongly in }H_{1}%
^{2}\left(  M\right)  \text{ and }L^{s}\left(  M\right)  \text{ for all
}s<2^{\sharp}\text{ as }q\rightarrow2^{\sharp}\medskip\\
\left(  c\right)  & w_{\gamma,q}\rightarrow w\text{ a.e in }M\text{ as
}q\rightarrow2^{\sharp}%
\end{array}
\]
\newline Now, we prove that the Lagrange multiplier $\lambda_{\gamma,q}$ is
bounded when $q$ tends to $2^{\sharp}$.

Using the definition of $\lambda_{\gamma,q}$ and the formula $\left(
\ref{eqn16}\right)  $ and the fact that
\[
\int\nolimits_{M}f\left\vert w_{\gamma,q}+h\right\vert ^{q-2}\left(
w_{\gamma,q}+h\right)  hdv_{g}\leq\gamma^{1-\frac{1}{q}}\left(  \int
\nolimits_{M}f\left\vert h\right\vert ^{q}dv_{g}\right)  ^{\frac{1}{q}}<\gamma
\]
we obtain%
\[%
\begin{array}
[t]{ll}%
0<\lambda_{\gamma,q} & =\frac{\int\nolimits_{M}\left(  P_{g}w_{\gamma
,q}\right)  w_{\gamma,q}dv_{g}}{\int\nolimits_{M}f\left\vert w_{\gamma
,q}+h\right\vert ^{q-2}\left(  w_{\gamma,q}+h\right)  w_{\gamma,q}dv_{g}%
}\medskip\\
& =\frac{\int\nolimits_{M}\left(  P_{g}w_{\gamma,q}\right)  w_{\gamma,q}%
dv_{g}}{\gamma-\int\nolimits_{M}f\left\vert w_{\gamma,q}+h\right\vert
^{q-2}\left(  w_{\gamma,q}+h\right)  hdv_{g}}\medskip\\
& \leq\frac{I\left(  t_{q}\psi_{1}\right)  }{\gamma-\gamma^{1-\frac{1}{q}%
}\left(  \int\nolimits_{M}f\left\vert h\right\vert ^{q}dv_{g}\right)
^{\frac{1}{q}}}<C\left(  \gamma,h\right)  .
\end{array}
\]
Since $\left(  \lambda_{\gamma,q}\right)  _{q}$ is bounded, there is a
subsequence of $\left(  \lambda_{\gamma,q}\right)  _{q}$ still labelled
$\left(  \lambda_{\gamma,q}\right)  _{q}$ which converges to $\lambda.$

Passing to the limit in equation (\ref{eqn14}), we obtain that $w$ the weak
limit of the sequence ($w_{\gamma,q}$)$_{q}$ is a weak solution of the
critical equation (\ref{eqn4}).
\end{proof}

Let $u:=w+h.$ If $\left(  \phi_{1},\phi_{2}\right)  \not \equiv \left(
0,0\right)  ,$ then $u\not \equiv 0$ is a non trivial solution of the equation%
\[
\left\{
\begin{array}
[c]{ll}%
P_{g}u=\lambda f\left\vert u\right\vert ^{2^{\sharp}-2}u & \text{in }M\\
u=\phi_{1}\text{ and }\partial_{\nu}u=\phi_{2} & \text{on }\partial M
\end{array}
\right.  \text{.}%
\]

And if $\left(  \phi_{1},\phi_{2}\right)  \equiv\left(  0,0\right)  ,$ then
$h\equiv0$ i.e. $u=w.$ We will prove that under some condition, $u$ is non
trivial solution of the equation%
\begin{equation}
\left\{
\begin{array}
[c]{ll}%
P_{g}u=\lambda f\left\vert u\right\vert ^{2^{\sharp}-2}u & \text{in }M\\
u=\partial_{\nu}u=0 & \text{on }\partial M
\end{array}
\right.  \label{eqn17}%
\end{equation}

\begin{proposition}
\label{prop2} Suppose that the minimizing sequence $\left(  w_{\gamma
,q}\right)  _{q}$ converges weakly to $w$ and put $\mu=\lim\limits_{q}%
\mu_{\gamma,q}$. Assume that%
\begin{equation}
\mu<\frac{\gamma^{\frac{2}{2^{\sharp}}}}{K_{0}\left\Vert f\right\Vert
_{\infty}^{^{\frac{2}{2^{\sharp}}}}} \label{18}%
\end{equation}
then $w$ is non trivial solution of the equation $\left(  \text{\ref{eqn17}%
}\right)  $.
\end{proposition}

\begin{proof}
First, we have%
\[
\gamma^{\frac{2}{q}}=\left(  \int_{M}f\left\vert w_{\gamma,q}\right\vert
^{q}\right)  ^{\frac{2}{q}}\leq\left\Vert f\right\Vert _{\infty}^{\frac{2}{q}%
}Vol_{g}\left(  M\right)  ^{\frac{2}{q}-\frac{2}{2^{\sharp}}}\left(  \int
_{M}\left\vert w_{\gamma,q}\right\vert ^{2^{\sharp}}\right)  ^{\frac
{2}{2^{\sharp}}}%
\]
Using the Sobolev inequality given by Lemma \ref{lem1}, for any $\epsilon>0,$
there exists $B_{\epsilon}>0$ such that%
\[%
\begin{array}
[c]{l}%
\gamma^{\frac{2}{q}}\left\Vert f\right\Vert _{\infty}^{-\frac{2}{q}}%
Vol_{g}\left(  M\right)  ^{\frac{2}{2^{\sharp}}-\frac{2}{q}}\leq\left(
K_{0}+\epsilon\right)  \left\Vert \Delta w_{\gamma,q}\right\Vert _{2}%
^{2}+B_{\epsilon}\left\Vert w_{\gamma,q}\right\Vert _{2}^{2}\medskip\\
\leq\left(  K_{0}+\epsilon\right)  \left[  \left(  1+\bar{\eta}\right)
\left\Vert \Delta w_{\gamma,q}\right\Vert _{2}^{2}-\bar{\eta}\left\Vert \Delta
w_{\gamma,q}\right\Vert _{2}^{2}\right]  +B_{\epsilon}\left\Vert w_{\gamma
,q}\right\Vert _{2}^{2}%
\end{array}
\]
where $\bar{\eta}$ is some small enough constant.

Since%
\[
\left\Vert \Delta w_{\gamma,q}\right\Vert _{2}^{2}=\mu_{\gamma,q}-\int
_{M}\left(  A\left(  \left(  \nabla w_{\gamma,q}\right)  ^{\#},\left(  \nabla
w_{\gamma,q}\right)  ^{\#}\right)  +aw_{\gamma,q}^{2}\right)  dv_{g}%
\]
it follows that%
\[%
\begin{array}
[c]{l}%
\gamma^{\frac{2}{q}}\left\Vert f\right\Vert _{\infty}^{-\frac{2}{q}}%
Vol_{g}\left(  M\right)  ^{\frac{2}{2^{\sharp}}-\frac{2}{q}}\medskip\\
\leq\left(  K_{0}+\epsilon\right)  \left\{  \left(  1+\bar{\eta}\right)
\left[  \mu_{\gamma,q}-\int_{M}\left(  A\left(  \left(  \nabla w_{\gamma
,q}\right)  ^{\#},\left(  \nabla w_{\gamma,q}\right)  ^{\#}\right)
+aw_{\gamma,q}^{2}\right)  dv_{g}\right]  -\bar{\eta}\left\Vert \Delta
w_{\gamma,q}\right\Vert _{2}^{2}\right\}  +\medskip\\
B_{\epsilon}\left\Vert w_{\gamma,q}\right\Vert _{2}^{2}\text{.}%
\end{array}
\]
Because $A$ is smooth, then for any $\eta>0$ there exists $C\left(
\eta\right)  >0$ such that%
\[
\left\vert \int_{M}A\left(  \left(  \nabla w_{\gamma,q}\right)  ^{\#},\left(
\nabla w_{\gamma,q}\right)  ^{\#}\right)  dv_{g}\right\vert \leq\eta\left\Vert
\Delta w_{\gamma,q}\right\Vert _{2}^{2}+C\left(  \eta\right)  \left\Vert
w_{\gamma,q}\right\Vert _{2}^{2}%
\]
we obtain%
\[%
\begin{array}
[c]{l}%
\gamma^{\frac{2}{q}}\left\Vert f\right\Vert _{\infty}^{-\frac{2}{q}}%
Vol_{g}\left(  M\right)  ^{\frac{2}{2^{\sharp}}-\frac{2}{q}}-\left(
K_{0}+\epsilon\right)  \left(  1+\bar{\eta}\right)  \mu_{\gamma,q}\medskip\\
\leq\left(  K_{0}+\epsilon\right)  \left\{  \left(  1+\bar{\eta}\right)
\left[  \eta\left\Vert \Delta w_{\gamma,q}\right\Vert _{2}^{2}+C\left(
\eta\right)  \left\Vert w_{\gamma,q}\right\Vert _{2}^{2}+\left\Vert
a\right\Vert _{\infty}\left\Vert w_{\gamma,q}\right\Vert _{2}^{2}\right]
-\bar{\eta}\left\Vert \Delta w_{\gamma,q}\right\Vert _{2}^{2}\right\}
\medskip\\
+B_{\epsilon}\left\Vert w_{\gamma,q}\right\Vert _{2}^{2}%
\end{array}
\]
taking $\eta>0$ such that
\[
\bar{\eta}=\frac{\eta}{1-\eta}%
\]
we get%
\[
\gamma^{\frac{2}{q}}\left\Vert f\right\Vert _{\infty}^{-\frac{2}{q}}%
Vol_{g}\left(  M\right)  ^{\frac{2}{2^{\sharp}}-\frac{2}{q}}-\left(
K_{0}+\epsilon\right)  \left(  1+\bar{\eta}\right)  \mu_{\gamma,q}\leq
C\left(  \epsilon,\eta\right)  \left\Vert w_{\gamma,q}\right\Vert _{2}^{2}%
\]
When $q$ tends to $2^{\sharp}$, the constants $\epsilon,$ $\eta$ are chosen
sufficiently small and if%
\[
\mu<\frac{\gamma^{\frac{2}{2^{\sharp}}}}{K_{0}\left\Vert f\right\Vert
_{\infty}^{\frac{2}{2^{\sharp}}}}%
\]
we infer that%
\[
\left\Vert w\right\Vert _{2}^{2}\geq C^{\prime}>0\text{.}%
\]
Hence $w\not \equiv 0$.
\end{proof}

Now we are going to establish the regularity of the solution of equation
$\left(  \ref{eqn6}\right)  $. We adapt the technique developed by Van der
Vorst \cite{26}, Djadli-Hebey-Ledoux \cite{10} and Esposito-Robert \cite{15},
for fourth-order elliptic equation.

First, we will enumerate two facts that will be useful to us. The first one is
the boudary version of the theorem given by F. Robert in \cite{25}.

\begin{theorem}
\cite{25} \label{th2} Let $\left(  M,g\right)  $ be a compact Riemannian
manifold with boundary of dimension $n\geq5$, $a\in C^{\infty}\left(
M\right)  $ and let $A$ be smooth symmetric $\left(  2,0\right)  $ tensor on
$M$. Assume that the operator $P_{g}=\Delta_{g}^{2}-\operatorname{div}%
_{g}A\left(  \nabla.\right)  ^{\#}+a$ is a coercive, then for any $f\in
H_{k,0}^{p}\left(  M\right)  $ there exists a unique $u\in H_{k+4,0}%
^{p}\left(  M\right)  $ such that $P_{g}u=f$. Moreover we have%
\[
\left\Vert u\right\Vert _{H_{k+4,0}^{p}\left(  M\right)  }\leq C\left\Vert
f\right\Vert _{H_{k,0}^{p}\left(  M\right)  }%
\]

\end{theorem}

\begin{proof}
The proof is similar to the one given in \cite{25}, we omit it.
\end{proof}

\begin{theorem}
\cite{3} Let $\left(  M,g\right)  $ be a compact Riemannian manifold of
dimension $n\geq1$ and with boundary, $p\geq1$ and $0\leq m\leq k$ two
integers such that $n\geq p\left(  k-m\right)  ,$ then $H_{k}^{p}\left(
M\right)  $ is embedded in $H_{m}^{q}\left(  M\right)  ,$ where $\frac{1}%
{q}=\frac{1}{p}-\frac{k-m}{n}$.
\end{theorem}

Our regularity theorem states as follows

\begin{theorem}
Let $\left(  M,g\right)  $ be a compact Riemannian manifold of dimension
$n\geq5$ and with boundary$.$ Assume that the operator $P_{g}=\Delta_{g}%
^{2}-\operatorname{div}_{g}A\left(  \nabla.\right)  ^{\#}+a$ is coercive. Let
$u\in H_{2,0}^{2}\left(  M\right)  $ be a weak solution of equation $\left(
\text{\ref{eqn6}}\right)  $, then $u\in C^{4}\left(  M\right)  $ and $u$ is a
strong solution of the equation $\left(  \text{\ref{eqn6}}\right)  $.
\end{theorem}

\begin{proof}
We follow closely the proof of Episoto-Robert \cite{15} and also that given by
Djadli-Hebey-Ledoux \cite{10} The method goes back to Van der Vorst \cite{26}.
For any $p>1$, denote by $L_{0}^{p}\left(  M\right)  $ the set of $f\in
L^{p}\left(  M\right)  $ with support in $M.$ Let $u\in H_{2,0}^{2}\left(
M\right)  $ be a weak solution of $\left(  \text{\ref{eqn6}}\right)  $ we
claim as in \cite{15} that, for any $\epsilon>0$ there exists $q_{\epsilon}\in
L_{0}^{\frac{n}{4}}\left(  M\right)  $, $f_{\epsilon}\in L_{0}^{\infty}\left(
M\right)  $ such that
\begin{align*}
\left(  \Delta_{g}+1\right)  ^{2}u  &  =\operatorname{div}_{g}\left(
A^{\#}du\right)  +\left(  1-a\right)  u+2\Delta_{g}u+\lambda f\left\vert
u\right\vert ^{2^{\sharp}-2}u\\
&  =b+q_{\epsilon}u+f_{\epsilon}%
\end{align*}
where $b=\operatorname{div}_{g}\left(  A^{\#}du\right)  +\left(  1-a\right)
u+2\Delta_{g}u.$\newline According to theorem \ref{th2}, for any $q>1$ and any
$f\in L_{0}^{q}\left(  M\right)  $, there exists a unique $u\in H_{4,0}%
^{q}\left(  M\right)  $ such that\newline%
\[
\left(  \Delta_{g}+1\right)  ^{2}u=f
\]
with
\[
\left\Vert u\right\Vert _{H_{4}^{q}\left(  M\right)  }\leq\left\Vert
f\right\Vert _{L^{q}\left(  M\right)  }.
\]

Now, we consider the following operator%
\[%
\begin{array}
[c]{ccc}%
H_{\epsilon}:u\in L_{0}^{q}\left(  M\right)  & \rightarrow & \left(
\Delta_{g}+1\right)  ^{-2}\left(  q_{\epsilon}u\right)  \in L_{0}^{q}\left(
M\right)
\end{array}
\]
with%
\begin{align*}
\left\Vert H_{\epsilon}u\right\Vert _{q}  &  =O\left(  \left\Vert \left(
\Delta_{g}+1\right)  ^{-2}\left(  q_{\epsilon}u\right)  \right\Vert
_{q}\right)  =O\left(  \left\Vert \left(  \Delta_{g}+1\right)  ^{-2}\left(
q_{\epsilon}u\right)  \right\Vert _{H_{4}^{\hat{q}}\left(  M\right)  }\right)
\\
&  \leq C\left\Vert q_{\epsilon}u\right\Vert _{\hat{q}}\leq C\left\Vert
q_{\epsilon}\right\Vert _{\frac{n}{4}}\left\Vert u\right\Vert _{q}\leq
C\epsilon\left\Vert u\right\Vert _{q}%
\end{align*}
and $\hat{q}=\frac{nq}{n+4s}$.

Hence, for $\epsilon>0$ and sufficiently small%
\[
\left\Vert H_{\epsilon}\right\Vert _{L^{q}\rightarrow L^{q}}\leq
C\epsilon<\frac{1}{2}\text{.}%
\]
\newline So, the operator
\[%
\begin{tabular}
[c]{l}%
$\left(  Id-H_{\epsilon}\right)  :L_{0}^{q}\left(  M\right)  \rightarrow
L_{0}^{q}\left(  M\right)  $\\
$\left(  Id-H_{\epsilon}\right)  u=\left(  \Delta_{g}+1\right)  ^{-2}\left(
b+f_{\epsilon}\right)  $%
\end{tabular}
\ \
\]
is an invertible, and we get $b+f_{\epsilon}\in L^{2}\left(  M\right)  ,$
hence $\left(  \Delta_{g}+1\right)  ^{-2}\left(  b+f_{\epsilon}\right)  \in
H_{4,0}^{2}\left(  M\right)  $.

From the Sobolev theorem, we deduce that%
\[
u\in L_{0}^{\frac{2n}{n-8}}\left(  M\right)  \text{, }f\left\vert u\right\vert
^{2^{\sharp}-2}u\in L_{0}^{^{\frac{2n}{\left(  n-8\right)  \left(  2^{\sharp
}-1\right)  }}}\left(  M\right)  =L_{0}^{\frac{2n\left(  n-4\right)  }{\left(
n-8\right)  \left(  n+4\right)  }}\left(  M\right)  .
\]
Since $\frac{2n\left(  n-4\right)  }{\left(  n-8\right)  \left(  n+4\right)
}>2,$ we obtain%
\[
\left(  \Delta_{g}+1\right)  ^{-2}u\in L_{0}^{2}\left(  M\right)  \text{.}%
\]
We now use a boostrap argument to construct an increasing sequence $\left(
p_{i}\right)  $ such that $u\in H_{4,0}^{p_{i}}\left(  M\right)  $ for all
$i\in%
\mathbb{N}
$\newline We let $p_{0}=2,$ the Sobolev's theorem asserts that
\[
u\in L^{\frac{np_{i}}{n-4p_{i}}}\left(  M\right)  \text{ and }f\left\vert
u\right\vert ^{2^{\sharp}-2}u\in L^{\frac{np_{i}}{\left(  n-4p_{i}\right)
\left(  2^{\sharp}-1\right)  }}\left(  M\right)  =L^{\frac{np_{i}\left(
n-4\right)  }{\left(  n-4p_{i}\right)  \left(  n+4\right)  }}\left(  M\right)
\]
then
\[
\left(  \Delta_{g}+1\right)  ^{-2}u\in L_{0}^{p_{i+1}}\left(  M\right)
\]
where
\[
p_{i+1}=\left\{
\begin{array}
[c]{cc}%
\frac{np_{i}\left(  n-4\right)  }{\left(  n-4p_{i}\right)  \left(  n+4\right)
} & \text{if }p_{i}<\frac{n}{4}\\
+\infty & \text{if }p_{i}\geq\frac{n}{4}%
\end{array}
\right.  \text{.}%
\]
We can verify by recurrence that for all $i\in%
\mathbb{N}
$ that $p_{i}>\frac{2n}{n+4},$ hence the sequence $\left(  p_{i}\right)  _{i}$
is increasing and bounded, consequently it converges to $l\geq2$ fulfilling
the relation%
\[
l=\frac{nl\left(  n-4\right)  }{\left(  n-4l\right)  \left(  n+4\right)
}\text{.}%
\]
\newline The last equation gives $l=\frac{2n}{n+4}$ which is a contradiction.
Hence $p_{i}\rightarrow+\infty$ and $u\in H_{4,0}^{p}\left(  M\right)  $ for
all $p>1$. Applying again the Sobolev's theorem, we get $u\in C^{4}\left(
M\right)  $. Hence $u$ is a strong solution to the critical equation
(\ref{eqn4}).
\end{proof}

\section{Test functions}

In this section, we prove that the condition (\ref{eqn18}) in the proposition
\ref{prop1} holds by using test functions.

For this purpose we consider a normal geodesic coordinate system $\left(
y^{1},y^{2},...,y^{n}\right)  $ centred at a point $x_{0}$ where $f$ reaches
its maximum. Denote by $S\left(  r\right)  $ the geodesic sphere centred at
$x_{0}$ and of radius $r$ $\left(  r<d=\text{ the injectivity radius}\right)
$ .\newline Let $d\Omega$ be the volume element of the Euclidean unit sphere
$S^{n-1}\left(  1\right)  $ and put
\[
G\left(  r\right)  =\frac{1}{\omega_{n-1}}\int\nolimits_{S\left(  r\right)
}\sqrt{\left\vert g\left(  x\right)  \right\vert }d\Omega
\]
where $\omega_{n-1}$ is the volume of $S^{n-1}\left(  1\right)  $ and
$\left\vert g\left(  x\right)  \right\vert $ is the determinant of the
Riemannian metric $g$.

The Taylor's expansion of $G\left(  r\right)  $ in a neighborhood of $r=0$ is
given by%
\[
G\left(  r\right)  =1-\frac{R\left(  x_{0}\right)  }{6n}r^{2}+o\left(
r^{2}\right)
\]
where $R\left(  x_{0}\right)  $ denotes the scalar curvature of $M$ at $x_{0}%
$.\newline Let $B\left(  x_{0},\delta\right)  $ be the geodesic ball of radius
$\delta$ centred at $x_{0}$ such that $0<2\delta<d$ and $\eta$ a smooth
function given by%
\[
\eta\left(  x\right)  =\left\{
\begin{array}
[c]{ll}%
1 & \text{if }x\in B\left(  x_{0},\delta\right) \\
0 & \text{if }x\in M-B\left(  x_{0},2\delta\right)
\end{array}
\right.  .
\]
\newline We consider the following radial smooth function:%
\[
u_{\epsilon}\left(  r\right)  =\frac{\eta\left(  r\right)  }{\left(
r^{2}+\epsilon^{2}\right)  ^{\frac{n-4}{2}}}%
\]
where $r=d\left(  x_{0},x\right)  $ denotes the geodesic distance to the point
$x_{0}.$\newline To simplify the computations, for any real positive numbers
$p,$ $q$ such that $p-q>1$, we define the following functions, $\left(
\text{see \ \cite{2}}\right)  $%
\[
I_{p}^{q}=\int\nolimits_{0}^{+\infty}\frac{t^{q}}{\left(  1+t\right)  ^{p}%
}dt\text{.}%
\]
The following relations result are immediate%
\[%
\begin{array}
[c]{cc}%
I_{p+1}^{q}=\frac{p-q-1}{p}I_{p}^{q}\text{,} & I_{p+1}^{q+1}=\frac{q}%
{p-q-1}I_{p+1}^{q}\text{.}%
\end{array}
\]
First, we compute the Taylor's expansion of the quotient:%
\[
Q_{\epsilon}=\frac{\mu\left(  u_{\epsilon}\right)  }{\left(  \gamma\left(
u_{\epsilon}\right)  \right)  ^{\frac{2}{2^{\sharp}}}}%
\]
where
\begin{equation}
\mu\left(  u_{\epsilon}\right)  =\int\nolimits_{M}\left(  \left(  \Delta
_{g}u_{\epsilon}\right)  ^{2}+A^{\#}\left(  du_{\epsilon},du_{\epsilon
}\right)  +au_{\epsilon}^{2}\right)  dv_{g} \label{eqn22}%
\end{equation}%
\begin{equation}
\gamma\left(  u_{\epsilon}\right)  =\int\nolimits_{M}f\left\vert u_{\epsilon
}\right\vert ^{2^{\sharp}}dv_{g}. \label{eqn23}%
\end{equation}
Now, we will compute each term of the expansion as it has been done \cite{7},
\cite{15} and \cite{25}. In the case $n>6$, we get
\[
\int\nolimits_{M}\left(  \Delta_{g}u_{\epsilon}\right)  ^{2}dv_{g}%
=\frac{\left(  n-4\right)  \omega_{n-1}I_{n}^{\frac{n}{2}-1}}{2\epsilon^{n-4}%
}\left\{  n\left(  n^{2}-4\right)  -\frac{n\left(  n^{2}+4n-20\right)
}{6\left(  n-6\right)  }R\left(  x_{0}\right)  \epsilon^{2}+o\left(
\epsilon^{2}\right)  \right\}
\]
and also%
\[
\int\nolimits_{M}A^{\#}\left(  du_{\epsilon},du_{\epsilon}\right)
dv_{g}=\frac{\left(  n-4\right)  \omega_{n-1}I_{n}^{\frac{n}{2}-1}}%
{2\epsilon^{n-4}}\left\{  \frac{4\left(  n-1\right)  }{n-6}Tr_{g}A\left(
x_{0}\right)  \epsilon^{2}+o\left(  \epsilon^{2}\right)  \right\}  .
\]
And Finally, we have%
\[
\int\nolimits_{M}au_{\epsilon}^{2}dv_{g}=\frac{1}{\epsilon^{n-4}}%
O(\epsilon^{4})\text{.}%
\]
Bearing the different terms of $\mu\left(  u_{\epsilon}\right)  $ in equation
(\ref{eqn22}), we obtain
\[%
\begin{array}
[c]{l}%
\mu\left(  u_{\epsilon}\right)  =\frac{n\left(  n-4\right)  \left(
n^{2}-4\right)  \omega_{n}}{2^{n}\epsilon^{n-4}}\times\\
\left\{  1+\frac{1}{n\left(  n^{2}-4\right)  \left(  n-6\right)  }\left(
4Tr_{g}A\left(  x_{0}\right)  \left(  n-1\right)  -\frac{n\left(
n^{2}+4n-20\right)  }{6}R\left(  x_{0}\right)  \right)  \epsilon^{2}+o\left(
\epsilon^{2}\right)  \right\}
\end{array}
\]
where $\omega_{n}=2^{n-1}I_{n}^{\frac{n}{2}-1}\omega_{n-1}$.

The Taylor's expansion of $\gamma\left(  u_{\epsilon}\right)  $ is given by
\[
\gamma\left(  u_{\epsilon}\right)  =\int\nolimits_{M}f\left\vert u_{\epsilon
}\right\vert ^{2^{\sharp}}dv_{g}=\frac{f\left(  x_{0}\right)  \omega_{n}%
}{2^{n}\epsilon^{n}}\left\{  1-\frac{1}{6\left(  n-2\right)  }\left(
\frac{3\Delta f\left(  x_{0}\right)  }{f\left(  x_{0}\right)  }+R\left(
x_{0}\right)  \right)  \epsilon^{2}+o\left(  \epsilon^{2}\right)  \right\}
\text{.}%
\]
Consequently
\[
\left(  \gamma\left(  u_{\epsilon}\right)  \right)  ^{-\frac{2}{2^{\sharp}}%
}=\frac{\left(  f\left(  x_{0}\right)  \right)  ^{-\frac{2}{2^{\sharp}}}%
\omega_{n}^{-\frac{2}{2^{\sharp}}}}{2^{4-n}\epsilon^{4-n}}\left\{
1+\frac{n-4}{6n\left(  n-2\right)  }\left(  R\left(  x_{0}\right)
+\frac{3\Delta f\left(  x_{0}\right)  }{f\left(  x_{0}\right)  }\right)
\epsilon^{2}+o\left(  \epsilon^{2}\right)  \right\}  \text{.}%
\]
Therefore the Taylor's expansion of $Q_{\epsilon}$, is given by
\begin{align*}
Q_{\epsilon}  &  =\frac{1}{\left(  f\left(  x_{0}\right)  \right)  ^{-\frac
{2}{2^{\sharp}}}K_{0}}\left\{  1+\frac{1}{2n\left(  n^{2}-4\right)  \left(
n-6\right)  }\times\right. \\
&  \left.  \left(  \left(  n+2\right)  \left(  n-4\right)  \left(  n-6\right)
\frac{\Delta f\left(  x_{0}\right)  }{f\left(  x_{0}\right)  }+8\left(
n-1\right)  Tr_{g}A\left(  x_{0}\right)  -4\left(  n^{2}-2n-4\right)  R\left(
x_{0}\right)  \right)  \epsilon^{2}+o\left(  \epsilon^{2}\right)  \right\}
\end{align*}
where
\[
\frac{1}{K_{0}}=\frac{n\left(  n-4\right)  \left(  n^{2}-4\right)  \omega
_{n}^{\frac{4}{n}}}{16}.
\]

It is obvious that if
\[
\left(  n+2\right)  \left(  n-4\right)  \left(  n-6\right)  \frac{\Delta
f\left(  x_{0}\right)  }{f\left(  x_{0}\right)  }+8\left(  n-1\right)
Tr_{g}A\left(  x_{0}\right)  -4\left(  n^{2}-2n-4\right)  R\left(
x_{0}\right)  <0
\]
we have
\[
Q_{\epsilon}<1\text{.}%
\]
In the case $n=6$, we have
\[
\int\nolimits_{M}\left(  \Delta_{g}u_{\epsilon}\right)  ^{2}dv_{g}%
=\frac{\left(  n-4\right)  ^{2}\omega_{n-1}}{2\epsilon^{n-4}}\left\{
\frac{n\left(  n^{2}-4\right)  }{n-4}I_{n}^{\frac{n}{2}-1}-\frac{2}{n}R\left(
x_{0}\right)  \epsilon^{2}\ln\frac{1}{\epsilon^{2}}+o\left(  \epsilon
^{2}\right)  \right\}
\]
and also%
\[
\int\nolimits_{M}A^{\#}\left(  du_{\epsilon},du_{\epsilon}\right)
dv_{g}=\frac{\left(  n-4\right)  ^{2}\omega_{n-1}I_{n}^{\frac{n}{2}-1}%
}{2\epsilon^{n-4}}\left(  \frac{Tr_{g}A\left(  x_{0}\right)  }{n}\epsilon
^{2}\ln\frac{1}{\epsilon^{2}}+o\left(  \epsilon^{2}\right)  \right)  .
\]
The expression of the last term is written as%
\[
\int\nolimits_{M}au_{\epsilon}^{2}dv_{g}=\frac{1}{\epsilon^{n-4}}%
O(\epsilon^{4})
\]
Inserting the different terms of $\mu\left(  u_{\epsilon}\right)  $ in
equation $\left(  \text{\ref{eqn22}}\right)  $, we get%
\begin{align*}
\mu\left(  u_{\epsilon}\right)   &  =\int\nolimits_{M}\left(  \left(
\Delta_{g}u_{\epsilon}\right)  ^{2}+A^{\#}\left(  du_{\epsilon},du_{\epsilon
}\right)  +au_{\epsilon}^{2}\right)  dv_{g}\\
&  =\frac{n\left(  n-4\right)  \left(  n^{2}-4\right)  \omega_{n}}%
{2^{n}\epsilon^{n-4}}\left\{  1+\frac{n-4}{\left(  n^{2}-4\right)
I_{n}^{\frac{n}{2}-1}}\left(  Tr_{g}A\left(  x_{0}\right)  -2R\left(
x_{0}\right)  \right)  \epsilon^{2}\ln\frac{1}{\epsilon^{2}}+o\left(
\epsilon^{2}\right)  \right\}
\end{align*}
where $\omega_{n}=2^{n-1}I_{n}^{\frac{n}{2}-1}\omega_{n-1}$.

In the same way as in the previous case, we obtain:%
\[
\left(  \gamma\left(  u_{\epsilon}\right)  \right)  ^{-\frac{2}{2^{\sharp}}%
}=\frac{\left(  f\left(  x_{0}\right)  \right)  ^{-\frac{2}{2^{\sharp}}}%
\omega_{n}^{-\frac{2}{2^{\sharp}}}}{2^{4-n}\epsilon^{4-n}}\left\{
1+\frac{n-4}{6n\left(  n-2\right)  }\left(  R\left(  x_{0}\right)
+\frac{3\Delta f\left(  x_{0}\right)  }{f\left(  x_{0}\right)  }\right)
\epsilon^{2}+o\left(  \epsilon^{2}\right)  \right\}  \text{.}%
\]

Finally, the Taylor expansion of $Q_{\epsilon}$, when $n=6,$ is given by%
\[
Q_{\epsilon}=\frac{1}{\left(  f\left(  x_{0}\right)  \right)  ^{-\frac
{2}{2^{\sharp}}}K_{0}}\left\{  1+\frac{n-4}{\left(  n^{2}-4\right)
I_{n}^{\frac{n}{2}-1}}\left(  Tr_{g}A\left(  x_{0}\right)  -2R\left(
x_{0}\right)  \right)  \epsilon^{2}\ln\frac{1}{\epsilon^{2}}+o\left(
\epsilon^{2}\right)  \right\}  \text{.}%
\]
Assuming
\[
Tr_{g}A\left(  x_{0}\right)  <2R\left(  x_{0}\right)
\]
we get
\[
Q_{\epsilon}<1\text{.}%
\]

\end{document}